\documentclass[11pt]{article}
\usepackage[affil-it]{authblk}

\newcommand\keywords[1]%
  {\begin{flushleft}
   \let\and\\%
   \textbf{Keywords:}\\
   #1
   \end{flushleft}%
  }

\title{Geometric Structure and Ergodic Properties of Bony Multi-Graphs }%

\author[1]{M. Rabiee}
\author[2]{F. H. Ghane}
\author[3]{M. Zaj}

\affil[1,2,3]{Department of Mathematics, Ferdowsi University of Mashhad, Mashhad, Iran}
\date{}

\makeatletter
\newcommand{\subjclass}[2][1991]{%
  \let\@oldtitle\@title%
  \gdef\@title{\@oldtitle\footnotetext{#1 \emph{Mathematics Subject Classification.} #2}}%
}


\usepackage[T1]{fontenc}
\usepackage{geometry}
\usepackage{xcolor}
\usepackage{listings}
\usepackage{fancyvrb}
\usepackage{amsmath, amsthm,amscd, amsfonts, amssymb, graphicx, color,float}
\usepackage{subfig}
\usepackage{listings}

\usepackage{color} 
\definecolor{mygreen}{RGB}{28,172,0}
\definecolor{almond}{rgb}{0.94, 0.87, 0.8}

\definecolor{electriccrimson}{rgb}{1.0, 0.0, 0.25}

\definecolor{mylilas}{RGB}{170,55,241}
\usepackage[bookmarksnumbered, colorlinks, plainpages]{hyperref}
\usepackage{lipsum}
\usepackage{fancyhdr}

\usepackage{algorithm,algorithmic}

\newtheorem{theorem}{Theorem}[subsection]
\newtheorem{definition}{Definition}[subsection]

\newtheorem{proposition}{Proposition}[subsection]
\newtheorem{corollary}{Corollary}[subsection]

\newtheorem{example}{Example}[subsection]
\newtheorem{remark}{Remark}[subsection]
\newtheorem{lemma}{Lemma}[subsection]

\numberwithin{equation}{subsection}
\let\lvert=|\let\rvert=|

\newcommand{\f}{{f}_{\theta}}

\begin{document}
\maketitle


\begin{abstract}
The main goal in this paper is to describe the geometric structure of invariant graphs of a certain
class of skew products.
Our focus is on attracting multi-graphs.
An invariant multi-graph is an invariant compact set which is a finite union of invariant graphs, and thus consists of a finite number of points on each fiber.
We introduce invariant bony multi-graphs and construct an open set of skew products over an invertible base map (solenoid map) having attracting invariant multi-graphs
and bony multi-graphs which support finitely many ergodic SRB measures. In this study some thermodynamic properties are investigated.
Finally, we extend our results to a family of skew products over a generalized baker map.
\end{abstract}
\maketitle

\section{\textbf{Introduction}}
The aim of this paper is to give a comprehensive description,
from both measure-theoretic and topological viewpoints, of the dynamics of skew product systems
that have either monotone interval maps.
A skew product system is a dynamical system $(\Theta \times \mathbb{X} , F )$ of the form
\begin{equation}\label{skew}
 F: \Theta \times \mathbb{X} \to \Theta \times \mathbb{X}, \quad (\theta,x)\mapsto (S(\theta), f_\theta(x)),
\end{equation}
that driven by a base map $S$ (which can be a solenoid map or a baker map).

In this context, describing the asymptotic behavior of the orbits and understanding how this behaviour changes when the system
is modified, are two main targets.
In the study of skew products, invariant graphs, particulary attracting invariant graphs, play an essential role. In fact,
they are the natural substitutes of fixed points and considerably simplify the dynamics of the forced systems.
When skew product systems have uniformly contracting fiber maps (hyperbolic setting), there exist invariant
attracting sets for the overall dynamics, which are the graph of continuous functions
(see \cite{HPS, HPS1}).
For the case that the fiber maps fail to be hyperbolic, we need to impose specific conditions which guarantees the existence of an attracting invariant graph
that attracts orbits almost surely. This includes
the skew product map possesses a negative Lyapunov exponent in the fibre direction or it satisfies the contraction on average condition \cite{Stark1, Stark2, ZG}.


Some related results on the ergodic properties and stability of attracting graphs under deterministic perturbations have been proposed by Campbell and Davies \cite{CD, C2}.
Also, sufficient conditions for the existence and
regularity of invariant graphs has appeared in \cite{Stark1, Stark2, SS}. Stark discussed \cite{Stark1, Stark2} a number of applications to the filtering
of time series, synchronization and quasiperiodically forced systems.
In general, synchronization is the phenomenon that different oscillations in coupled systems will converge to oscillations that move with identical frequency. However, external forcing, rather than coupling, can also synchronize dynamics. This phenomenon has been widely studied in
theoretical physics \cite{PR, RS} and also recently in mathematics \cite{H}.
Another applications of invariant graphs in many branches of nonlinear dynamics were also proposed (e.g. \cite{BHN, CD, DK, HOY, PC, SD} etc.).

For skew-products with monotone fiber maps there is a close
relation between (the existence of) invariant graphs, (fiber) Lyapunov exponents and ergodic measures of $F$ (e.g. \cite{JJ, J, ZG, FKG} etc.).
For instance, the stability of an invariant graph is determined by its Lyapunov exponent,
if it is negative then the graph is attracting.
 In this literature, there are attracting invariant graphs with more complicated dynamics.
This includes bony graph attractors which are currently object of intense study \cite{KV, Kud, ZG}.
A bony graph attractor is an attracting invariant graph which intersects almost every fiber at a single point, and
any other fiber at an interval which is called a bone.

Here, a more general version of an invariant graph, the so-called invariant multi-graph is considered.
An invariant multi-graph is an invariant compact set which is
a finite union of invariant graphs, and thus consists of a finite number of points on each
fibre. In \cite{JK}, J\"{a}ger and Keller provided a criteria, in terms of Lyapunov exponents, for the existence of
attracting invariant multi-graphs.
Gelfert and Oliveira \cite{GO} studied step skew-products over a finite-state shift (base) space
whose fiber maps are $C^1$ injective maps on the unit interval. They provided
certain invariant sets having  a multi-graph structure and can be written as graphs
of one, two or more functions defined on the base.

As an objective, we describe the geometrical structures of attracting invariant
multi-graphs for a certain class of skew product systems as defined by (\ref{skew}).
In particular, we construct robust invariant bony multi-graphs.

Other goal is the existence of SRB measures whose supports lie on invariant graphs.
SRB stands for Sinai, Ruelle and Bowen, are the invariant measures most compatible with volume when
volume is not preserved.
References \cite{KV, VY, ZG} contain further results about the existence
and finiteness of SRB measures supported on attracting invariant graphs.
Another main objective is to show that the attracting multi-graphs, in our setting, carry finitely many ergodic SRB measures.

We also investigate some thermodynamic properties for these systems.
Thermodynamic formalism, that is the formalism of equilibrium statistical physics, was
adapted to the theory of dynamical systems in the works of Sinai, Ruelle and Bowen \cite{B,R2,S}.
Topological pressure, topological entropy and equilibrium states are the fundamental tools in thermodynamic formalism.
The existence and uniqueness of equilibrium states are currently object of intense study.
Here, we will provide some sufficient conditions ensuring the existence of equilibrium states supported on invariant multi-graphs.
Our approach for existence of equilibrium states is based on the technique applied in \cite{RV}. In this scenario, we require the graph functions. In our setting, since the contraction on the fibers is non-uniform, graph functions are only upper semicontinuous.
For existence of equilibrium states we apply a version of variational principle provided by Rauch \cite{R}
for discontinuous potentials.


\textbf{This work is organized as follows:}
First, in Sect. 2, we state some terminology concerning invariant multi-graphs
 and recall conditions ensuring the existence of multi-graphs from \cite{J,JJ,JK}.
In Sect. 3, we construct (Theorem \ref{t144}, Theorem \ref{thm88} and Corollary \ref{cor33}) an open set of skew products given by (\ref{skew})
 over an invertible base map (solenoid map) having attracting invariant multi-graphs or bony multi-graphs
that carry finitely many ergodic SRB measures.
Then we investigate some thermodynamic properties in Sect. 4.
Sufficient conditions ensuring the existence of equilibrium states supported on invariant multi-graphs is presented (Proposition \ref{p00}).
Finally, in Sect. 5, Theorem \ref{thm777}, we provide a measure theoretical isomorphism between the skew products over the solenoid map and skew products over a generalized baker map and deduce the existence of invariant bony multi-graphs for these systems.
\section{\textbf{Preliminaries}}
In this section, we introduce the concepts and the notations which are basic for the study
in the following sections.
\begin{definition}[SRB measure]\label{def:skew}
Let $X$ be a manifold and $f:X\longrightarrow X$ a continuous map.
An $f$-invariant measure $\mu$ is called $Sinai$-$Ruelle$-$Bowen$ (SRB) measure (or a $physical$ measure) if there exists a mesurable set $E\subset X$, of positive Lebesgue-measure, such that for every continuous function $\varphi:X\longrightarrow \mathbb{R}$ and every $x\in E$ we have:
\begin{align}
\lim_{n\rightarrow\infty}\frac{1}{n}\sum\limits_{j=0}^{n}\varphi(f^{j}(x))=\int\limits_{X}\varphi d\mu.
\end{align}
The set $E$ is called the basin of $\mu$.
In other words, time averages of all continuous functions are given by the corresponding space averages computed with respect to $\mu$, at least for a large set of initial states $x\in X.$
\end{definition}
\subsection{Skew products}
Consider a skew product system $(\Theta \times \mathbb{X} , F )$ of the form (\ref{skew}), where
the dynamics on the fibre space $\mathbb{X}$ may be interpreted as being driven by another system
$(\Theta, S)$ since the transformations $f_\theta : \mathbb{X} \mapsto \mathbb{X}$ depend on $\theta$.
 On the other hand, the space $\mathbb{X}$ is principally considered
as the fibre space over the basis dynamics $(\Theta , S)$, i.e. the fibre map $f_\theta$ can be considered as a map
from $\{\theta\} \times \mathbb{X}$ to $\{S\theta\} \times \mathbb{X}$, where $\{\theta\} \times \mathbb{X}$ is the fibre space over $\theta \in \Theta$.
We adopt the usual notation $ F^{n}(\theta,x ) = (S^{n}\theta , f_{\theta}^{n}(x))$,
for the iterates $F^{n}$ of $F$, where $f^{n}_{\theta} = f_{S^{n-1}(\theta)} \circ \cdots \circ f_{\theta}$.
Hence, $f_\theta^{n+k}(x) = f_{S^k \theta}^n (f^k_\theta(x))$. For $n = 1$ and $k =-1$ this includes the identity and $f_\theta^{-1}(x) = (f_{S^{-1}\theta})^{-1}(x)$.

In this article, we are interested in skew product dynamical systems $F$, defined as below:
\begin{definition}\label{generalskew}
$\mathcal{F}$ denotes the family of all skew product transformations $F$ given by (\ref{skew}) on $\Theta \times \mathbb{I}$ having the following properties:
\begin{enumerate}
  \item [(1)] The base $\Theta$ is equipped with a $\sigma$-algebra $\mathcal{B}$ and a probability measure $m$, such that $(\Theta, \mathcal{B}, m)$
becomes a probability measurable space, and the base transformation $S: \Theta \to \Theta$ is a bi-measurable and ergodic
measure-preserving bijection.
 In most situations we will assume that $\Theta$ is a compact metric
space, in which case $\mathcal{B}$ is always the (completed) Borel $\sigma$-algebra and $S$ is a homeomorphism.
  \item  [(2)] Let $\mathbb{I} = [0,1]$ and $\text{int}(\mathbb{I})=(0,1)$. The fibre maps $ f_{\theta} : \mathbb{I} \to \text{int}(\mathbb{I}) $ are given by $f_{\theta}(x) = \pi_{2} \circ F(\theta,x) $  with $\pi_{2}$ the natural projection from $\Theta \times \mathbb{I}$ to $\mathbb{I}$. We will assume that the fiber maps $f_{\theta} $ are $C^2$ increasing interval maps.
  \item [(3)] For each $x \in \mathbb{I}$ the map $\theta \mapsto F(\theta,x)$ is measurable. If $\Theta$ is a compact metric
space and $S$ is a homeomorphism then the map $\theta \mapsto F(\theta,x)$ is continuous.
 \end{enumerate}
\end{definition}
\begin{remark}\label{rem:com}
By definition, the dynamic contracts the whole space, i.e. $F(\Theta \times \mathbb{I}) \subset \Theta \times \text{int}(\mathbb{I})$.
\end{remark}
We equip $\mathcal{F}$ with the following metric:
\begin{align} \label{m}
 \text{dist}_{\mathcal{F}}(F,G):=\sup\limits_{\theta\in \Theta}\text{dist}_{C^2}(f^{\pm 1}_\theta,g^{\pm 1}_\theta), \ \ \text{for each}\ \ F,G \in \mathcal{F},
\end{align}
where $f_{\theta}$ and $g_{\theta}$ are the fiber maps of $F$ and $G$ respectively.
\subsection{Invariant graphs and Lyapunov exponents}
Invariant graphs are fundamental objects in the study of skew product systems,
and they are of major interest.
\begin{definition}[Invariant graph]\label{def:graph}

Let $F \in \mathcal{F}$ be a skew product as Definition \ref{generalskew}. A measurable function $\gamma : \Theta \to \mathbb{I}$ is called an invariant graph (with respect to $F$) if for all $\theta \in \Theta$:
\begin{equation}\label{in}
F(\theta , \gamma(\theta)) = (S \theta , \gamma(S \theta )) ,    \text{ equivalently} \quad f_{\theta}(\gamma(\theta)) = \gamma(S\theta).
\end{equation}
The point set $\Gamma := \{ (\theta , \gamma(\theta))  : \theta \in \Theta \}$ will be called invariant graph as well, but it is labeled with the corresponding capital letter. Denote by $\text{Cl}(\Gamma)$ the closure of $\Gamma$ in $\Theta \times \mathbb{I}$.

Let $m$ be an invariant measure for the base map $S$. We say that $\gamma: \Theta \to \mathbb{I}$ is an $(F,m)$ invariant graph if
\begin{equation}\label{1graph}
 f_\theta(\gamma(\theta))=\gamma(S(\theta)), \ \text{for} \ m-a.e. \ \theta \in \Theta.
\end{equation}
\end{definition}

In the sequel, we will denote by $\pi_1 : \Theta \times \mathbb{I} \to \Theta$ and
$\pi_2 : \Theta \times \mathbb{I} \to \mathbb{I}$ the canonical projections onto the first and second coordinates,
respectively.
To any $(F, m)$-invariant graph $\gamma$, an $F$-invariant measure $m_\gamma$ can be assigned by defining
\begin{equation}\label{m}
  m_\gamma(A) = m (\pi_1(A \cap \Gamma))
\end{equation}
for any measurable set $A \subseteq \Theta \times \mathbb{I}$.
Note that the measure $m_\gamma$ is ergodic if and only if $m$ is ergodic.

Throughout the paper, the set of all $F$-invariant probability measures on $\Theta \times \mathbb{I}$
is denoted by $\mathcal{M}(F)$, and the set of $\mu \in \mathcal{M}(F)$ which project to $m$ by $\mathcal{M}_m(F)$.

The next theorem is a counterpart of Theorem 1.8.4 in \cite{AL} to our setting, see also \cite[Theorem 1.1]{JJ} and Furstenberg Theorem \cite{F}.
It provides that there is a one-to-one correspondence
between invariant graphs and invariant ergodic measures of skew products forced by monotone interval maps.

\begin{theorem}\label{thm00}
Suppose $F \in \mathcal{F}$ is a skew product as Definition \ref{generalskew}, $m \in \mathcal{M}(S)$ and $\mu \in  \mathcal{M}_m(F)$ is ergodic.
Then there exists an $(F, m)$-invariant graph $\gamma$ such that $\mu = m_\gamma$.
\end{theorem}

In the investigation of skew product systems,
attracting invariant graphs are often useful characteristics, which we simply call attractors. Formally, we introduce the following definition. Note that the effect of the attraction is observed only
in the fibre space $\mathbb{I}$.
\begin{definition}[Attracting invariant graphs]
A point $(\theta,x) \in \Theta \times \mathbb{I}$ is attracting, if there is a constant $\delta >0$ such that
$$\lim_{n\to \infty}|\f^n(x)-\f^n(z)| = 0,$$
for all $z \in (x-\delta,x+\delta)$.
Furthermore, an invariant graph $\gamma$ is called an attracting invariant graph or attractor with
respect to the invariant probability $\mu$ if $\mu$-almost every point is attracting.
\end{definition}
In a weaker form, we say that $A$ is an \emph{attractor} for $F$ (in the sense of Milnor \cite{Mil}) if there is a set of points in the phase space with positive probability whose future orbits tend to $A$, as the number of iterates tends to infinite. The set of orbits attracted to $A$ in the future is called its \emph{basin} and denoted by $B(A)$.
\begin{definition}[Maximal attractor]\label{def:att}
Given a continuous skew product $F$,
a compact set $D \subset \Theta \times \mathbb{I} $ is trapping if $F(D)\subset int{D}$. Then the closed $F$-invariant set
$$A_{max}:=\bigcap \limits _{n\geq 0}F^{n}(D)$$
is said to be a maximal attractor for $F$.
\end{definition}

\begin{definition}[Maximal Lyapunov exponent]\label{max}
The maximal Lyapunov exponent of an $F$-invariant probability measure $\mu$ is defined by
\begin{equation}\label{lya1}
 \lambda(\mu,F)=\lim_{n \to \infty}\frac{1}{n}\int_{\Theta \times \mathbb{I}} \log | Df_\theta^{n}(x)| d\mu(\theta,x).
\end{equation}
\end{definition}
The maximal Lyapunov exponent of an invariant graph $\gamma$ with respect to a $S$-invariant probablity measure $m$ is defined as
\begin{equation}\label{lya2}
 \lambda(m,\gamma)=\lim_{n \to \infty}\frac{1}{n}\int_{\Theta} \log |Df_\theta^{n}(x)|dm(\theta).
\end{equation}
Note that (\ref{lya2}) is a special case of (\ref{lya1}), with $\mu$ given by $\mu(A) = m(\{\theta \in \Theta: (\theta,\gamma(\theta))\in A \}$.
\begin{definition}[Upper Lyapunov exponent]
Following \cite{FKG, JJ}, for each skew product $F$ given by Definition \ref{generalskew},  the \emph{upper Lyapunov exponent of} $(\theta, x) \in \Theta \times \mathbb{I}$ is
\begin{equation}\label{lya3}
 \lambda_{max}(\theta, x)=\lim_{n \to \infty}\frac{1}{n} \log |Df_\theta^n(x)|,
\end{equation}
when the limit exists, where $Df_\theta(x)$ is the derivative of $f_\theta$ in $x$.

Given any $F$-invariant probability measure $\mu$, we define the \emph{upper Lyapunov exponent of} $\mu$ by
\begin{equation}\label{lya4}
\lambda_{max}(\mu)=\int \lambda_{max}(\theta, x) d\mu (\theta, x).
\end{equation}
\end{definition}
Note that since $f_\theta^n$ is increasing interval maps, $|Df_\theta^n(x)|=Df_\theta^n(x)$.

\begin{definition}[Fiber Lyapunov exponent]\label{ver}
The (fiber) Lyapunov exponent of an $(F ,m)$-invariant graph $\gamma$ is given by
\begin{equation}\label{lya}
 \lambda_m(\gamma)=\int_\Theta \log Df_\theta(\gamma(\theta))dm(\theta).
\end{equation}
\end{definition}
\begin{remark}\label{att}
Let $\gamma$ be an $(F ,m)$-invariant graph with $ \lambda_m(\gamma)<0$. Then, by \cite[Lemma 1.10]{JJ}, $\gamma$ is an attracting graph with respect to the invariant
probability $m_\gamma$.
\end{remark}
Note that by the Birkhoff ergodic theorem
\begin{equation*}
\begin{aligned}
\lambda_{max}(\theta,\gamma(\theta))
&= \lim_{n\to \infty} \frac{1}{n} \log Df_{\theta}^{n}(\gamma(\theta)) = \lim_{n\to \infty} \frac{1}{n} \sum_{k=0}^{n-1}\log Df_{S^{k}\theta}(f_{\theta}^{k}(\gamma(\theta)))\\
&= \lim_{n \to \infty}\frac{1}{n} \sum_{k=0}^{n-1}\log Df_{S^{k}\theta}(\gamma(S^k\theta))
= \int_{\Theta} \log Df_{\theta}(\gamma(\theta))dm (\theta)  \\
&=\lambda_{m}(\gamma)
\end{aligned}
\end{equation*}
for $m$-a.e.~$\theta\in\Theta$.
So the average Lyapunov exponent of an invariant graph equals its point-wise Lyapunov exponent for $m$-a.e. $\theta\in \Theta$.

Here, we focus on skew products having the
non-uniformly contraction rates along the fiber. We address the situation where we only
have information about average rates of contraction. This is a
weaker form of contraction which is a necessary condition for synchronization, see \cite{Stark2}.
When we talk about average rates of contraction, we need an invariant measure. In our setting, take the $S$-invariant measure $m$, as in Definition \ref{generalskew}.
\begin{definition}[Non-uniform contraction condition]\label{av}
Suppose that the limit
\begin{equation}\label{lya}
\lambda(\theta)=\lim_{n \to \infty}\frac{1}{n}\sup_{x,x^{\prime}\in I}\log \frac{d(f_\theta^n(x), f_\theta^n(x^{\prime}))}{d(x,x^{\prime})}
\end{equation}
exists for almost every $\theta$ and is a measurable function of $\theta$ (this is the maximal Lyapunov exponent in the fibre direction
of the skew product $F$).
We say that the skew-product $F$ contracts non-uniformly if there exists $\lambda < 0$ such that
\begin{equation}\label{neg}
 \lambda(\theta)\leq \lambda < 0
\end{equation}
for a.e. $\theta \in \Theta$.
\end{definition}
\begin{remark}\label{r11}
In a special case, when $\mu$ is given by $\mu(A) = m(\{\theta \in \Theta: (\theta,\gamma(\theta))\in A \}$, if the fiber Lyapunov exponent of $m$ has negative upper bound
then the skew product $F$ satisfies (\ref{neg}) and hence it contracts non-uniformly.
\end{remark}
As a consequence of non-uniform contraction condition we have the following result
of Stark \cite{Stark2}, which applies in a more general setting.
\begin{theorem}(\cite[Theorem 1.4]{Stark2})\label{thmg}
Suppose $\Theta$ is a compact metric space, $S:\Theta \to \Theta$ a homeomorphism,
$m$ an invariant measure, $X$ a complete metric space and the fiber map $f: \Theta \times X \to X$ is a continuous map
satisfying (\ref{neg}). Then there exists an $S$-invariant set $\Lambda \subset \Theta$ such that $m(\Lambda)= 1$ and a
function $\gamma: \Lambda \to X$ such that the graph of $\gamma$ is invariant and attracting under $F$.
\end{theorem}
\subsection{Multi graphs}
We concentrate on the case that a compact invariant set is
just a finite union of invariant graphs, and thus consists of a finite number of points on each
fibre.

Let $F \in \mathcal{F}$ be a skew product over a base map $S$ as Definition \ref{generalskew}.
\begin{definition}[Multi-graph]
Following \cite{GO}, given $F \in \mathcal{F}$, a multi-function $\psi: D \subset \Theta \to \mathbb{I}$ is a relation that associates
to every point $\theta \in D$ a nonempty subset $\psi(\theta) \subset \mathbb{I}$. A multi-function $\psi: D \subset \Theta \to \mathbb{I}$ is
uniformly finite if
there exists $k \geq 1$ such that $\# \psi(\theta) \leq k$ for all $\theta \in D$.
Given a uniformly finite multi-function $\psi$, we call the set $\{(\theta, \psi(\theta)): \theta \in D\}$
a multi-graph in $\Theta \times \mathbb{I}$.
\end{definition}

J\"{a}ger proved that \cite[Theorem 1.14]{JJ} for skew products with $C^1$ interval fiber maps having a minimal homeomorphism as a base, strict negativity of the Lyapunov exponents on a compact invariant set implies that
this set is a multi graph. Then this extended to a more general case \cite[Theorem 1.2]{JK}.

 Let $F \in \mathcal{F}$ be a skew product over a base map $S$ as Definition \ref{generalskew}.
 Then the tuple $(\Theta,\mathcal{B},m,S)$ is a measure-preserving dynamical system.

We say \cite{JK} $K\subset \Theta \times \mathbb{I}$ is a \emph{random compact set} if
\begin{enumerate}
  \item $K_\theta=\{x \in \mathbb{I}: (\theta,x)\in K\}$ is compact for all $\theta \in \Theta$;
  \item the functions $\theta \mapsto d(x,K_\theta)$ are measurable for all $x \in \mathbb{I}$.
 \end{enumerate}
The set of all $f$-invariant probability measures $\mu$ which project to $m$ supported on a compact $F$-invariant set $K$ is denoted by $\mathcal{M}_m^K(F)$.
 \begin{theorem}\cite[Theorem 1.2]{JK}
 Let $F \in \mathcal{F}$ be a skew product with base $(\Theta,\mathcal{B},m,S)$, the family
 $(x \mapsto \log \|D f_\theta^k(x)\|)_{\theta \in \Theta}$ is equicontinuous for all $k \in \mathbb{N}$ and $K \subset \Theta \times \mathbb{I}$
is a random compact set such that the maximal Lyapunov exponent $\lambda(\mu,F)$ given by (\ref{lya1}) is negative for all $\mu \in \mathcal{M}_m^K(F)$.
Then there exists an integer $n$ such that $\# K_\theta=n$ for $m$-a.e. $\theta \in \Theta$.
 \end{theorem}
If the base map $S$ is a homeomorphism on a compact metric space $\Theta$, since fiber maps are $C^2$ and $F$ is continuous,
the next corollary can be followed immediately.
\begin{corollary}\label{thmm}
Suppose $\Theta$ is a compact metric space, $m$ an invariant ergodic measure and $S: \Theta \to \Theta$ is an
ergodic measure-preserving homeomorphism.
Assume, $F \in \mathcal{F}$ is a skew product as Definition \ref{generalskew} over the base map $S$ and $K$ is a compact $F$-invariant set.
Further, assume that for all $\mu \in \mathcal{M}_m^K(F)$ we have
$\lambda(\mu,F) < 0$. Then there exists an integer $n$ such that $\# K_\theta=n$ for $m$-a.e. $\theta \in \Theta$. In particular, if $n > 1$ then $K$ is a multi-graph.
\end{corollary}
\section{\textbf{Bony multi-graphs for skew products with interval fiber maps}}
Our major goal here is to describe the structure of invariant graphs and study the geometry of attractors, mainly, in the case that the basis dynamical system is a
solenoid map. In particular, we construct robust bony multi-graphs. In our construction, first we provide a single skew product $\widetilde{F}$ over an expanding circle map.
Then, we consider its extension $F$ which is a skew product over the solenoid map and show that every skew product $G$
which is close enough to $F$ admits a bony multi-graph which carries finitely many ergodic SRB measures.

We recall the family $\mathcal{F}$ of skew products given by Definition \ref{generalskew}. The next theorem is the main result of this section.
\begin{theorem}\label{t144}
Given any $n \in \mathbb{N}$, there exists an open set $\mathcal{U}\subset \mathcal{F}$ of skew products of the form (\ref{skew}) with interval fiber maps having a solenoid map as the base map
such that each skew product system $G$ belonging to $\mathcal{U}$ admits an attracting invariant multi-graph or bony multi-graph having exactly $n$
SRB measures supported in $K(G)$.
\end{theorem}
The rest of this section is devoted to proving this theorem.
\subsection{Skew products forced by expanding circle maps}
If $X$ is a metric space and $f : X \longrightarrow X$ is a continuous map, then
we say that $f$ is \emph{weakly contractive} whenever for each $x,y \in X$ with $x\neq y$, $d(f(x),f(y)) < d(x,y)$.
\begin{remark}\label{fweak}
It is a well-known fact \cite[Coro.~3]{J} that if $f$ is weakly contractive and $X$ is compact then there exists a unique fixed point
$x \in X$ of the map $f$. Furthermore, for every $y \in X$, $\lim_{k \to \infty}f^k(y)=x$ uniformly.
Then we say that $x$ is a \emph{weak attracting fixed point}.
Clearly if $f$ is a weakly contractive map then
$$d(f^{n}(y), f^{n}(z)) \to 0, \ as, \ n \to \infty,$$
for each $y, z \in X$.
\end{remark}
\begin{definition}
We say that the interval map $f$ is $s$-weakly contractive if it satisfies the following conditions:
\begin{enumerate}
  \item [$(1)$] $f$ is a $C^1$ weakly contractive map.
  \item [$(2)$] Let $p$ be the unique fixed point of $f$. Then, $Df(p)=1$ and for all $x\neq 0$, one has $Df(x)<1$.
  \end{enumerate}
\end{definition}
The following example shows that the set of $s$-weakly contractive maps is nonempty.
\begin{example}
Define $f:[0.1,0.7] \to [0.1,0.7]$ by $f(x)=3.098 x^{1.83}-2.5 x^{2.4} +0.1$. Then $f$ is a weakly contractive map with the unique fixed point $0.466$.
In particular, $Df(0.466)=1$ and for all $x\neq 0.466$, we have $Df(x)<1$.
\end{example}
In \cite{E}, the author gave another examples of $s$-weakly contractive maps.
\begin{definition}[Weak-pair]\label{k}
Take two $C^2$ increasing interval maps $f_i$, $i=0, 1,$ defined on an interval $[a,b]$, $C^2$-close to the identity so that they fulfill the following conditions:
\begin{enumerate}
  \item [$(1)$] Each $f_i$, $i=0, 1,$ is a $s$-weakly contractive map having a (unique) weak attracting fixed point $p_i$.
  \item [$(2)$] For every $x \in [a,b]$, the following ``contraction on average property" hold:
  \begin{equation}\label{caverage}
    \sum_{i=0}^1 \log Df_i(x)<0.
  \end{equation}


  \item [$(3)$] The weak attracting fixed points $p_i$, $i=0, 1$, are pairwise disjoint. Moreover, $p_i\neq a,b$ and $g_i(p_j)\neq p_i$ for each $j \neq i$.

\end{enumerate}
 Let $p_0< p_1$ and $J = [p_0 , p_1]\subset [a,b]$. Then we say that $(f_0|_{J},f_1|_{J})$ is a weak-pair for $J$.
\end{definition}
\begin{remark}\label{cov}
By conditions (1) and (3) the following ``covering property" holds: there exists the points $x_0$ and $x_1$ with $p_0 < x_0 < x_1 < p_1$ such that for the interval
$B = (x_0, x_1) \subset \text{int}(J)$ one has:
\begin{equation*}
\forall x\in [x_1,x_2], \ Df_i(x)<1 \ \text{and} \   \text{Cl}(B) \subset f_0(B) \cup f_1(B).
\end{equation*}
\end{remark}
Note that, covering property introduced in \cite{NP}.
In Figure \ref{fig:1} below, the pairs $(f_0|_{I_1},f_1|_{I_1})$ and $(f_0|_{I_2},f_1|_{I_2})$ are weak-pairs.
\begin{figure}[h]
\begin{center}
\includegraphics[scale=0.35]{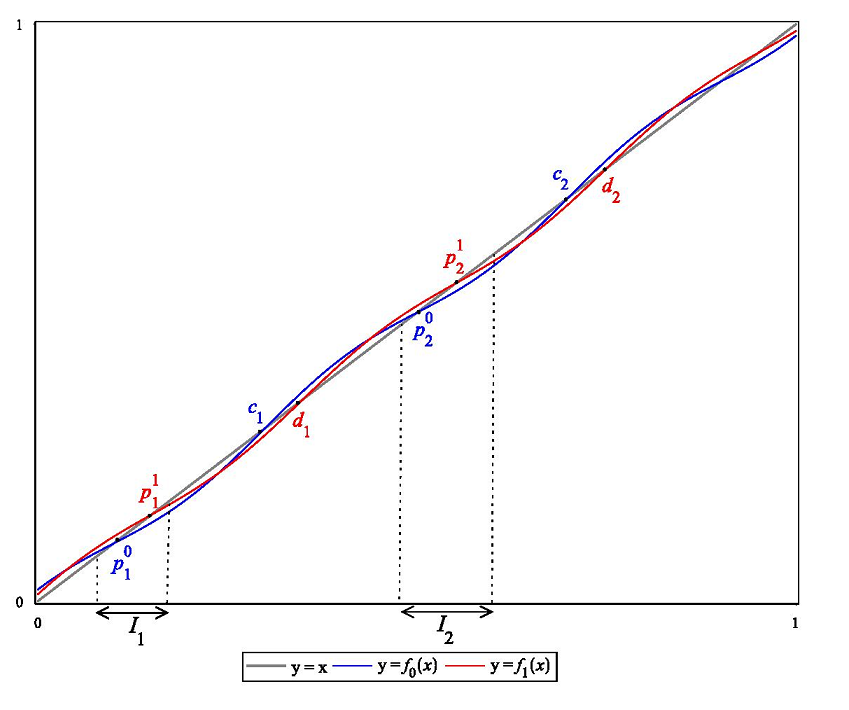}
\caption{Two increasing $C^2$ maps $f_0$ and $f_1$ with two weak-pairs}
\label{fig:1}
\end{center}
\end{figure}

Fix $n \geq 1$ and consider $n$ disjoint subintervals of $\mathbb{I}=[0, 1]$,
say $I_1 = [a_1, b_1], \ldots , I_n = [a_n, b_n]$, with
\begin{equation*}
0 <a_1 < b_1 < a_2 < \ldots < a_n < b_n <1
\end{equation*}
and two increasing $C^2$ maps $f_0, f_1: \mathbb{I} \to \mathbb{I}$ such that
\begin{enumerate}
  \item [($i$)] $f_0(I_i)\subset I_i, f_1(I_i)\subset I_i \ \text{and} \ (f_0|_{I_i},f_1|_{I_i}) \ \text{is \ a \ weak-pair \ for \ each} \ I_i, \ i=1, \ldots, n$,
  hence there exist weak attracting fixed points $p^0_i$ and $p^1_i$ of $f_0$ and $f_1$, respectively;
  \item [($ii$)] $f_0$ and $f_1$ have repelling fixed points $c_i$ and $d_i$, $i=1, \ldots, n-1$, respectively, such that
\begin{equation*}
  0 <b_1 < c_1 < d_1 < a_2 <\ldots b_{n-1}<c_{n-1}< d_{n-1}<a_n < b_n <1.
\end{equation*}
  \end{enumerate}
  We assume that $a_i<p^0_i<p^1_i<b_i$.
By Remark \ref{cov}, there exist points $x_i$ and $y_i$, $i=1, \ldots, n$,
 such that $a_i < x_i < y_i < b_i$ and intervals $B_i = (x_i , y_i ) \subset \text{int}(I_i )$ satisfies
the following covering property:
\begin{equation}\label{f1}
  \forall x \in [x_i,y_i], Df_i(x)<1, \ \text{and} \ \text{Cl}(B_i)\subset f_0(B_i) \cup f_1(B_i).
\end{equation}
 Consider the circle expanding map $\omega :\mathbb{T}^1 \to \mathbb{T}^1$, $\omega(t)=4t$ $(mod \ 1)$.
 Let
 \begin{equation*}
   L_i \subset \mathbb{T}^1, \ i=0,1,
 \end{equation*}
be disjoint closed arcs, with the length of each arc equal to $1/4$.
Then for $0 \leq i \leq 1$ we have $\omega(L_i) = \mathbb{T}^1$.
Take a smooth map
 $\ell : \mathbb{T}^1 \to [0, 1]$ such that $\ell |L_0\equiv0$, $\ell |L_1\equiv 1$ and outside the $\delta$-neighborhood of $L_0 \cup L_1$, for sufficiently small $\delta>0$, $\ell(t) \in (0,1)$.

 Now, consider an isotopy
\begin{align}\label{is}
f_{t}(x)=(1-(\ell(t)^2))f_{0}(x)+(\ell(t)^2)f_{1}(x)
\end{align}
between $f_0$ and $f_1$. Then, for each $t \in L_0$, $f_t=f_0$ and for each $t \in L_1$, $f_t=f_1$.
It is easy to see that
\begin{align}\label{de}
\forall t\in \mathbb{T}^1\setminus B_\delta(L_0 \cup L_1) \ \text{and}\ \forall x\in I_i, \ i=1, \cdots, n, \ \text{one has} \ Df_{t}(x) <1.
\end{align}

We now define a skew product over the expanding circle map $\omega$, corresponding to the fiber maps $f_0$ , $f_1$ by
\begin{align}\label{be}
\widetilde{F}:\mathbb{T}^1\times I\to \mathbb{T}^1\times I,\quad \widetilde{F}(t,x)=(\omega(t),f(t,x)),
\end{align}
where $f(t,x):=f_{t}(x)$ given by (\ref{is}).
In the rest of this article we fix the skew products $\widetilde{F}$ given by (\ref{be}).

Clearly, by construction, for each $1\leq i\leq n$, the compact set $D_i:=\mathbb{T}^1 \times I_i $ is a trapping region for $\widetilde{F}$.
Let us take
\begin{equation}\label{max}
 A_{max}(\widetilde{F}):=\bigcap_{n\geq 0}\widetilde{F}^n(\mathbb{T}^1 \times \mathbb{I}), \ \text{and} \ \Lambda_i:=\bigcap_{n\geq 0}\widetilde{F}^n(\mathbb{T}^1 \times I_i), \ 1\leq i\leq n.
\end{equation}
Then $\Lambda_i$, $1\leq i\leq n$, are maximal attractors of $\widetilde{F}$ corresponding to the trapping regions $D_i$ and the union $\bigcup_{i=1}^n \Lambda_i$ is a compact invariant set contained in the global attractor $A_{max}(\widetilde{F})$.
\subsection{Extension skew products and invariant graphs}
Since the base map $\omega:\mathbb{T}^1 \to \mathbb{T}^1$ is non-invertible, hence
 we can always find an extension $S: \Omega\to \Omega$ that is invertible.
By extension we mean that there exists a surjective map
$\textbf{p}: \Omega\to \mathbb{T}^1$ such that,
 $\textbf{p}\circ S=\omega\circ \textbf{p}$.
 To begin with, take $\Omega$ to be the set of all pre-orbits of
 $\omega$, that is, all sequences $(t_{-n})_{n\leq 0}$
 satisfying
 $\omega (t_{-n})=t_{-n+1}$ for every $n>0$.
 Consider the map $\textbf{p}: \Omega\to \mathbb{T}^1$ sending each sequence  $(t_{-n})_{n\geq 0}$
 to its term $t_0$ of order zero. Observe that $\mathbf{p}(\Omega)=\mathbb{T}^1$. Finally, we define
 $S: \Omega\to \Omega$ by
 $$ S(\dots, t_{-n}, \dots, t_0)= (\dots, t_{-n}, \dots, t_0, \omega(t_0)).$$
 It is clear that $S$ is well defined and satisfies
 $\mathbf{p}\circ S=\omega\circ\mathbf{p}$.
The inverse limit space $\Omega$ is endowed with the product topology, then it is easy to see that $S$ is a homeomorphism on $\Omega$.
\begin{remark}\label{e}
On ergodic point of view, given an ergodic measure $\mu^+$ defined on
Borel subsets of $\mathbb{T}^1$ there exists a unique measure $\mu$ defined on Borel
subsets of $\Omega$ such that $\mathbf{p}_\ast \mu=\mu^+$.
We mention that since $\omega:\mathbb{T}^1\longrightarrow \mathbb{T}^1$ is a $C^2$-expanding endomorphism then $\omega$ possesses an absolutely continuous invariant ergodic measure $\nu^{+}$ whose density is bounded and bounded away from zero. \cite{M}. Moreover, $\nu^{+}$ is equivalent to Lebesgue.
Thus, $(\Omega,S)$ has an invariant ergodic measure $\nu$ inherited from the invariant measure $\nu^+$ for $\omega$ on $\mathbb{T}^1$, i.e. $\mathbf{p}_\ast \nu=\nu^+$.
\end{remark}
For the skew products $\widetilde{F}$ given by (\ref{be}), we define an extension skew product map $F=F(\widetilde{F})$ on $\Omega \times \mathbb{I}$ by
\begin{equation}\label{sss}
F(\textbf{t},x)=(S(\textbf{t}),f(\textbf{t},x))=(S(\textbf{t}),f_{t_0}(x)),
\end{equation}
for each $\textbf{t}=(\dots, t_{-n}, \dots, t_0)\in \Omega$, where $f_{t_0}$ is given by (\ref{is}).
Note that the inverse map $F^{-1}$ is given by
$$F^{-1}(\mathbf{t},x)=(\dots, t_{-2}, t_{-1}, (f_{t_{-1}})^{-1}(x)).$$
Then $\mathbf{p}\times id$ is a semi conjugacy between $\widetilde{F}$ and $F$.
\begin{remark}\label{ef}
By construction, the fiber maps of $\widetilde{F}$ and its extension $F$ are the same.
\end{remark}
 In the rest of this paper, we fix the skew product $F$ with the fiber maps $f_{\textbf{t}}$, $\textbf{t}\in \Omega$.
By construction, for each $1\leq i\leq n$, the compact set $\Omega \times I_i $ is a trapping region for $F$.
Let us take
\begin{equation}\label{max}
 A_{max}(F):=\bigcap_{n\geq 0}F^n(\Omega \times \mathbb{I}), \ \text{and} \ \Delta_i:=\bigcap_{n\geq 0}F^n(\Omega \times I_i), \ 1\leq i\leq n.
\end{equation}
Then $\Delta_i$, $1\leq i\leq n$, are maximal attractors of $F$ corresponding to the trapping regions $\Omega \times I_i $ and the union $K:=\bigcup_{i=1}^n \Delta_i$ is a compact invariant set contained in the global attractor $A_{max}(F)$.

Let us take
\begin{equation}\label{rest}
F_i:=F|_{\Omega \times I_i}, \ i=1, \dots, n.
\end{equation}

\begin{lemma}\label{lemma0}
The skew product $F_i$ satisfies the non-uniformly contraction condition given by Definition \ref{av} for each, $i=1, \dots,n$.
\end{lemma}
\begin{proof}
By construction, for each $\textbf{t} \in \Omega$, the fiber map $f_\textbf{t}$ of $F$ is a $C^2$ diffeomorphism. Hence, there exists a $\varsigma \in L^1(\Omega,\nu)$
such that $d(f_\textbf{t}(x),f_\textbf{t}(x^{\prime}))<\varsigma(\textbf{t})d(x,x^{\prime})$, for $\nu$-almost every $\textbf{t}$. Then, by \cite[Lem.~5.1]{Stark2}, for all $m\geq 0$
\begin{equation}\label{n-lya}
\varsigma_{m,i}(\textbf{t})=\sup_{x,x^{\prime}\in I_i}\frac{d(f_\textbf{t}^m(x), f_\textbf{t}^m(x^{\prime}))}{d(x,x^{\prime})}, \ i=1, \cdots,n
\end{equation}
exists for almost every $\textbf{t}$ and it is measurable, where $f^m_{\textbf{t}}=f_{S^{m-1}(\textbf{t})}\circ \ldots \circ f_{S(\textbf{t})}\circ f_\textbf{t}$. Also, by \cite[Lem.~5.3]{Stark2}, the sequence $\varsigma_{m,i}$ is submultiplicative. It immediately follows that $\log \varsigma_{m,i}$ is subadditive, that is $\log \varsigma_{m+k,i}(\textbf{t})\leq \log \varsigma_{m,i}(\textbf{t})+ \log \varsigma_{k,i}(S^m(\textbf{t}))$.
We can thus apply the Kingman's subadditive ergodic theorem \cite[Thm.~5.4]{Kr} to
$\log \varsigma_{m,i}$ to deduce that the limit
\begin{equation}\label{lya}
\lambda_i(\textbf{t})=\lim_{m \to +\infty}\frac{1}{m}\log \varsigma_{m,i}(\textbf{t}),
\end{equation}
exists for $\nu$-almost every $\textbf{t}$ and is $S$ invariant. Moreover,
\begin{equation}\label{44}
\lim_{m \to \infty}\frac{1}{m}\textnormal{log} \varsigma_{m,i}(\textbf{t})= \inf_{m}\frac{1}{m}\textnormal{log} \varsigma_{m,i}(\textbf{t})
\end{equation}
 and it is constant by the ergodicity of $\nu$, denoted by $\lambda_i(\nu)$.
 Note that, by definition, for each
 $t=(\dots, t_{-1}, t_0)\in \Omega$, $$f^m_{\textbf{t}}=f_{S^{m-1}(\textbf{t})}\circ \dots \circ f_{S(\textbf{t})}\circ f_\textbf{t}=f_{\omega^{m-1}(t_0)}\circ \dots \circ f_{\omega(t_0)}\circ f_{t_0}.$$

On the other hand, by (\ref{de}), for each $\textbf{t}=(\ldots, t_{-1}, t_0) \in \Omega$ with $t_0 \in \mathbb{T}^1 \setminus B_\delta(L_0 \cup L_1)$ and every $x \in I_i$, we have $\|Df_\textbf{t}(x)\|<1$.
It is known that \cite{BY} for Lebesgue almost all $t_0$ the set $\{\omega^n(t_0): n \in \mathbb{N}\}$ is dense in $\mathbb{T}^1$. Since, $\nu^+$ is equivalent to Lebesgue, the same holds for $\nu^+$-almost
all $t_0$. Hence, for $\nu^+$-almost all $t_0$, there exists $n \in \mathbb{N}$ such that $\omega^n(t_0) \in \mathbb{T}^1 \setminus B_\delta(L_0 \cup L_1)$. This implies that $f_{\omega^n(t_0)}$
is a uniformly contracting map.
By these facts, since $\nu$ inherited from $\nu^+$, there exists an upper bound $\lambda <0$ such that $\lambda_i(\nu)\leq \lambda <0$ (note that for each $t$,
by definition, the fiber map $f_t$ is either weakly contractive or uniformly contracting).
Hence, $F_i$ satisfies the non-uniformly contraction condition for each $i=1, \dots, n$.
\end{proof}
Given a small $\varepsilon>0$, take a $\varepsilon$-neighborhood $\mathcal{U}$ of $F$ with respect to the metric defined by (\ref{m}).
Taking $\varepsilon$ small enough such that for each $G \in \mathcal{U}$ and for all $i=1, \dots, n$, the region $\Omega \times I_i$ is a trapping region for $G|_{\Omega \times I_i}$.
For each $G\in \mathcal{G}$, we take
\begin{equation}\label{rest1}
G_i:=G|_{\Omega \times I_i}, \ i=1, \dots, n.
\end{equation}
\begin{lemma}\label{cor0}
Let the $\varepsilon$-neighborhood $\mathcal{U}$ of $F$ defined above be small enough. Then for every perturbed skew product $G \in \mathcal{U}$, the restricted skew products $G_i$, $i=1, \dots,n$,
given by (\ref{rest1}) satisfy the non-uniformly contraction condition given by Definition \ref{av}.
\end{lemma}
\begin{proof}
By Lemma \ref{lemma0}, for each $i=1, \dots, n$, the skew product $F_i$ satisfies the non-uniformly contraction condition.
So, by (\ref{44}), for $\nu$-almost $\textbf{t}$, the limit $\lim_{m \to \infty}\frac{1}{m}\textnormal{log} \varsigma_{m,i}(\textbf{t})= \inf_{m}\frac{1}{m}\textnormal{log} \varsigma_{m,i}(\textbf{t})$
exists and it is constant,  by the ergodicity of $\nu$, denoted by $\lambda_i(\nu)$; in particular, $\lambda_i(\nu)<0$.
Take
\begin{equation}\label{la}
 \lambda=\max\{\lambda_i(\nu): i=1, \dots, n\}.
\end{equation}
Clearly, $\lambda <0$.
Take a skew product $G \in \mathcal{U}$. Let $G$ defined by $G(\textbf{t},x)=(S(\textbf{t}),\textbf{g}(\textbf{t},x))$. Then, by definition of the metric given by $(\ref{m})$,
$\sup_{\textbf{t}} dist_{C^2}(f_\textbf{t}^{\pm}, g_\textbf{t}^{\pm})< \varepsilon$.
Take
\begin{equation}\label{per}
\zeta_{m,i}(\textbf{t}):=\sup_{x,x^{\prime}\in I_i}\frac{d(g_\textbf{t}^m(x), g_\textbf{t}^m(x^{\prime}))}{d(x,x^{\prime})},
\end{equation}
for each $i=1, \dots, n$.
Since, for each $\textbf{t}$, the fiber map $g_\textbf{t}$ is a $C^2$ diffeomorphism, there exists a $\varsigma \in L^1(\Omega,\nu)$
such that $d(g_\textbf{t}(x),g_\textbf{t}(x^{\prime}))<\varsigma(\textbf{t})d(x,x^{\prime})$, for $\nu$-almost every $\textbf{t}$. Then, by \cite[Lem.~5.1]{Stark2}, for all $m\geq 0$,
$\zeta_{m,i}(\textbf{t})$ exists for almost every t and it is measurable.
Also, by \cite[Lem.~5.3]{Stark2}, the sequence $\zeta_{m,i}$ is submultiplicative and $\log \zeta_{m,i}$ is subadditive.
By the Kingman's subadditive ergodic theorem \cite[Thm.~5.4]{Kr}, the limit
\begin{equation}\label{lya}
\lambda_i(\textbf{t}, G)=\lim_{m \to +\infty}\frac{1}{m}\log \zeta_{m,i}(\textbf{t}),
\end{equation}
exists for $\nu$-almost every $\textbf{t}$ and is $S$ invariant. Moreover,
\begin{equation}\label{t44}
\lim_{m \to \infty}\frac{1}{m}\textnormal{log} \zeta_{m,i}(\textbf{t})= \inf_{m}\frac{1}{m}\textnormal{log} \zeta_{m,i}(\textbf{t})
\end{equation}
 and it is constant by the ergodicity of $\nu$, denoted by $\lambda_i(\nu, G)$.
 We need to see that this number is negative.

Fix small $\delta$ with $\lambda+\delta <0$.
By (\ref{44}) and (\ref{t44}), $\lambda_i(\nu)=\inf_{n}\frac{1}{m}\textnormal{log} \varsigma_{m,i}(\textbf{t})$ and $\lambda_i(\nu, G)=\inf_{m}\frac{1}{m}\textnormal{log} \zeta_{m,i}(\textbf{t})$, for $\nu$-a.e. $\textbf{t}$. Thus, for a typical point
$\textbf{t}$ for $\lambda_i(\nu)$ and $\lambda_i(\nu, G)$,
there exists $m_0=m_0(\textbf{t}) \in \mathbb{N}$
such that $\frac{1}{m_0}\textnormal{log} \varsigma_{m_0,i}(\textbf{t})< \lambda+\delta$, where $\lambda$ is given by (\ref{la}).
Taking $\varepsilon$ small enough, we get $\frac{1}{m_0}\textnormal{log} \zeta_{m_0,i}(\textbf{t})< \lambda+\delta <0$.
Consequently,
$$
\lambda_i(\nu, G)=\inf_{m}\frac{1}{m}\textnormal{log} \zeta_{m,i}(\textbf{t})< \frac{1}{m_0}\textnormal{log} \zeta_{m_0,i}(\textbf{t})< \lambda+\delta <0.
$$
Thus $G_i$ satisfies the non-uniformly contraction condition.
\end{proof}
Note that, for each $1\leq i\leq n$, the compact set $\Omega \times I_i $ is a trapping region for each $G \in \mathcal{U}$.
Let us take
\begin{equation}\label{maxi}
 A_{max}(G):=\bigcap_{n\geq 0}G^n(\Omega \times I), \ \text{and} \ \Delta_i(G):=\bigcap_{n\geq 0}G^n(\Omega \times I_i), \ 1\leq i\leq n.
\end{equation}
Then $\Delta_i(G)$, $1\leq i\leq n$, are maximal attractors corresponding to the trapping regions $\Omega \times I_i $ and the union $K(G):=\bigcup_{i=1}^n \Delta_i(G)$ is a compact invariant set contained in the global attractor $A_{max}(G)$.

Applying Theorem \ref{thmg} we get the next result (see also\cite[Thm.~5]{C2}, \cite[Thm.~1.4]{Stark2} or \cite[Pro.~2.3]{BHN}).
\begin{theorem}\label{thm000}
Let $G\in \mathcal{U}$, where the open set $\mathcal{U}$ is given by Lemma \ref{cor0}. Then there exist $S$-invariant sets $\Omega_i(G)\subseteq \Omega$, $i=1, \dots,n$, such that $\nu(\Omega_i(G))=1$
and measurable functions $\gamma_{G,i}:\Omega_i(G)\to I_i$ such that $\Gamma_{G,i}$, the graphs of $\gamma_{G,i}$, are invariant under $G$.
Furthermore $\Gamma_{G,i}$, $i=1, \dots,n$, are attracting $(G,\nu)$ invariant graphs.
\end{theorem}

For the skew product $F$  and any small perturbation $G \in \mathcal{U}$, we define measures $\nu_{\gamma_{F,i}}$ and $\nu_{\gamma_{G,i}}$ on $\Omega \times I_i$ by
\begin{equation}\label{meas}
 \nu_{\gamma_{F,i}}:=\nu\circ(id\times \gamma_{F,i})^{-1}|_{\Omega \times I_i}, \ \  \nu_{\gamma_{G,i}}:=\nu\circ(id\times \gamma_{G,i})^{-1}|_{\Omega \times I_i},
\end{equation}
for each $i=1, \dots,n$. Since $\nu$ is ergodic, so they are also ergodic. In particular, they are supported on the maximal attractors $\Delta_i(F)$ and $\Delta_i(G)$, respectively.

Note that for the case $n=1$, i.e. we have only one invariant graph, you can see more details of construction and the proofs in \cite{ZG}.

\begin{definition}[Bony attracting graph and bony multi-graph]
Let $G$ be a skew product with interval fiber maps and an invertible base
map $S$.
\begin{enumerate}
  \item [$(1)$] An attracting bony graph is an attracting invariant graph that is the union of the graph of a continuous function $\gamma$ defined on some set of full
measure and a set of vertical closed intervals (bones) contained in the closure of the
graph, see \cite{KV,ZG}. If the graph function $\gamma$ defined on the whole space $\Omega$, then we call the graph of $\gamma$ a continuous attracting invariant graph.
  \item [$(2)$] A compact invariant set $K$ is a bony multi-graph for $G$ if $K$ is a multi-graph composed of the finite union of attracting invariant graphs
and at least one of the invariant graphs contained in $K$ is an attracting bony graph.
 \end{enumerate}
 \end{definition}
 In \cite[Theorem~1]{ZG}, we investigated a certain class of skew products whose fiber maps are increasing $C^2$-interval maps over the base map $(\Omega,S)$.
We proved the existence of an open set of skew products such that any skew product $G$ belonging to this set admits a unique attracting invariant graph for which the
following dichotomy is ascertained. This invariant graph is either a continuous attracting graph or an attracting bony graph.
In both cases they carry an SRB measure. Here, we extend the result to the case that $G$ admits more than one attracting invariant graph.

The next result can be followed by applying \cite[Theorem 1]{ZG} to our setting by any modification (see also \cite[Lemma~2.6]{ZG}).
\begin{theorem}\label{thm88}
Given a skew product system $G$ belonging to $\mathcal{U}$, the maximal attractors $\Delta_i(G)$, $i=1, \dots, n$, defined by (\ref{maxi}), satisfy the following properties:
\begin{enumerate}
  \item [$(1)$] The maximal attractor $\Delta_i(G)$ is either a continuous attracting invariant graph or an attracting bony graph.
In the case $\Delta_i(G)$ is an attracting bony graph, the graph function $\gamma_{G,i}$ defined on a subset $\Omega_i(G)\subseteq \Omega$
with total measure and there exists a family of vertical closed intervals (bones), one bone in each
fiber $\textbf{t} \times I_i$, with $\textbf{t} \in \Omega \setminus \Omega_i(G)$; in particular, the bones are contained in the closure
of the graph $\Gamma_{G,i}$ and $\text{Cl}(\Gamma_{G,i})=\Delta_i(G)$.
  \item [$(2)$] The invariant ergodic measures $\nu_{\gamma_{G,i}}$, $i=1, \dots,n$, given by (\ref{meas}), supported on the closure of $\Gamma_{G,i}$, are SRB measures.
\end{enumerate}
\end{theorem}
However, weakly contractive fiber maps prevent the existence
of bones in the skew product, but bones appear in small perturbations of the original system.
By \cite[Lemma~2.5]{ZG}, we get the next lemma. It shows that the set of $G \in \mathcal{U}$ having an attracting bony graph is nonempty.
\begin{lemma}\label{G}
There exists a skew product $G \in \mathcal{U}$ and an index $1\leq i\leq n$ such that $\Delta_i(G)$ is the closure of an attracting bony graph.
In particular, the subset of bones has the cardinality of the continuum and is dense in the attractor.
\end{lemma}
\begin{lemma}\label{lem99}
There exists an upper semicontinuous extension of the graph function
$\gamma_{\mathbb{G},i}: \Omega_i(G) \to I_i$ to the whole space
$\Omega$.
\end{lemma}
\begin{proof}
By Theorem \ref{thm88}, there exist $S$-invariant sets $\Omega_i(G)$  such that $\nu(\Omega_i(G))=1$
and measurable function $\gamma_{G,i} : \Omega_i(G) \to I_i$ such that $\Gamma_{G,i} $, the graph of $\gamma_{G,i}$, is invariant under $G$.

Let $\textbf{t} \in \Omega$ and $\varepsilon > 0$ be given.
By Theorem \ref{thm88}, $\text{Cl}(\Gamma_{G,i})=\Delta_i(G)$.
 Take $I_{i,\textbf{t}}:=\{\textbf{t}\} \times I_i$ and $\Delta_{G,i,\textbf{t}}:=\Delta_i(G)\cap I_{i,\textbf{t}}$. Then
\begin{equation*}
\Delta_{G,i,\textbf{t}}=\bigcap_{n \geq 0}g_{\textbf{t}}\circ g_{S^{-1}(\textbf{t})}\circ \dots \circ g_{S^{-n}(\textbf{t})}(I_i)=\lim_{n \to \infty}g_{\textbf{t}}\circ g_{S^{-1}(\textbf{t})}\circ \dots \circ g_{S^{-n}(\textbf{t})}(I_i), \ i=1, \dots,n.
\end{equation*}
Note that $g_{\textbf{t}}\circ g_{S^{-1}(\textbf{t})}\circ \dots \circ g_{S^{-n}(\textbf{t})}(I_i)$ is a sequence of nested intervals, and thus $\Delta_{G,i,\textbf{t}}$ is either an interval or a
single point. In particular, if $\textbf{t} \in \Omega_i(G)$ then $\Delta_{G,i,\textbf{t}}$ is a single point.

 If $n$ is big enough then
 \begin{equation*}
 g_{\textbf{t}}\circ g_{S^{-1}(\textbf{t})}\circ \dots \circ g_{S^{-n}(\textbf{t})}(I_i) \subset U_{\frac{\varepsilon}{2}}(\Delta_{G,i,\textbf{t}}),
 \end{equation*}
where $U_{\frac{\varepsilon}{2}}(\Delta_{G,i,\textbf{t}})$ is $\frac{\varepsilon}{2}$-neighborhood of $\Delta_{G,i,\textbf{t}}$.
Let $\textbf{t}^{\prime}$ be sufficiently close to $\textbf{t}$. Then $g_{\textbf{t}^{\prime}} \circ g_{S^{-1}(\textbf{t}^{\prime})} \circ \dots \circ g_{S^{-n}(\textbf{t}^{\prime})}$ is $C^2$ close to $g_{\textbf{t}}\circ g_{S^{-1}(\textbf{t})}\circ \dots \circ g_{S^{-n}(\textbf{t})}$ and hence
\begin{equation*}
 g_{\textbf{t}^{\prime}}\circ g_{S^{-1}(\textbf{t}^{\prime})}\circ \dots \circ g_{S^{-n}(\textbf{t}^{\prime})}(I_i) \subset U_{\frac{\varepsilon}{2}}(\Delta_{G,i,\textbf{t}}).
 \end{equation*}
Then
\begin{equation}\label{1}
\Delta_{G,i,\textbf{t}^{\prime}}\subset g_{\textbf{t}^{\prime}} \circ g_{S^{-1}(\textbf{t}^{\prime})} \circ \dots \circ g_{S^{-n}(\textbf{t}^{\prime})}(I_i)\subset U_{\frac{\varepsilon}{2}}(g_{\textbf{t}}\circ g_{S^{-1}(\textbf{t})}\circ \dots \circ g_{S^{-n}(\textbf{t})}(I_i))\subset U_{\varepsilon}(\Delta_{G,i,\textbf{t}}).
\end{equation}
This implies the upper-semicontinuity of $\Delta_{G,i,\textbf{t}}$.
This semicontinuity will immediately imply the continuity of its graph part.

Indeed, by (\ref{1}), we obtain that
$$\text{diam}(\Delta_{G,i,\textbf{t}^{\prime}})\leq \text{diam}(\Delta_{G,i,\textbf{t}})+2\varepsilon. $$
If $\Delta_{G,i,\textbf{t}}$ is a single point, then $\text{diam}(\Delta_{G,i,\textbf{t}})=0$.
Then, by this fact and (\ref{1}),
$$|\text{diam}(\Delta_{G,i,\textbf{t}})-\text{diam}(\Delta_{G,i,\textbf{t}^{\prime}})|= \text{diam}(\Delta_{G,i,\textbf{t}^{\prime}})\leq
\text{diam}(\Delta_{G,i,\textbf{t}})+2\varepsilon \leq 2\varepsilon.$$
This implies continuity at $\textbf{t}$.

Consider an extension of $\gamma_{G,i} : \Omega_i(G) \to I_i$ to the whole space $\Omega$ as the following:
for each $\textbf{t} \in \Omega$, take
\begin{equation}\label{2}
 \gamma_{i}(\textbf{t})=\lim_{n \to \infty}g_{\textbf{t}}\circ g_{S^{-1}(\textbf{t})}\circ \dots \circ g_{S^{-n}(\textbf{t})}(a_i),
\end{equation}
 where $a_i$ is the left endpoint of $I_i$. By (\ref{1}), if $\textbf{t} \in \Omega_i(G)$, then
$\lim_{n \to \infty},g_{\textbf{t}}\circ g_{S^{-1}(\textbf{t})}\circ \dots \circ g_{S^{-n}(\textbf{t})}(I_i)=\lim_{n \to \infty},g_{\textbf{t}}\circ g_{S^{-1}(\textbf{t})}\circ \dots \circ g_{S^{-n}(\textbf{t})}(a_i) $.
This fact implies that the map $\gamma_i$ as defined in (\ref{2}) is an extension of $\gamma_{G,i}$.
We claim that $\gamma_i$ is upper semicontinuous.

Indeed, let $\gamma_i(\textbf{t})<c$, for some $c \in (0,1)$. Then, by definition, $\lim_{n \to \infty}g_{\textbf{t}}\circ g_{S^{-1}(\textbf{t})}\circ \dots \circ g_{S^{-n}(\textbf{t})}(a_i)<c$.
Hence, there exists $n_0 \in \mathbb{N}$ such that $g_{\textbf{t}}\circ g_{S^{-1}(\textbf{t})}\circ \dots \circ g_{S^{-n}(\textbf{t})}(a_i)<c$, for each $n \geq n_0$.
Note that the fiber maps are increasing interval maps. Take $\varepsilon >0$ small enough such that if an interval map $f$  is $C^2$-$\varepsilon$-close to $g_{\textbf{t}}\circ g_{S^{-1}(\textbf{t})}\circ \dots \circ g_{S^{-n}(\textbf{t})}$ then $f(a_0)<c$. Now take $\textbf{t}^{\prime}$ sufficiently close to $\textbf{t}$ such that $g_{\textbf{t}^{\prime}} \circ g_{S^{-1}(\textbf{t}^{\prime})} \circ \dots \circ g_{S^{-n}(\textbf{t}^{\prime})}$ is $C^2$-$\varepsilon$-close to $g_{\textbf{t}}\circ g_{S^{-1}(\textbf{t})}\circ \dots \circ g_{S^{-n}(\textbf{t})}$.
Then $g_{\textbf{t}^{\prime}} \circ g_{S^{-1}(\textbf{t}^{\prime})} \circ \dots \circ g_{S^{-n}(\textbf{t}^{\prime})}(a_0)<c$. This fact implies the upper semicontinuity of $\gamma_i$.
\end{proof}

By Corollary \ref{thmm}, Lemma \ref{cor0} and Theorem \ref{thm88}, we get the next result.
\begin{corollary}\label{cor33}
Let $G \in \mathcal{U}$ and consider the compact invariant set $K(G):=\bigcup_{i=1}^n \Delta_i(G)$, where $\Delta_i(G)$, $i=1, \dots,n$, are given by (\ref{maxi}).
Then $K(G)$ is an invariant multi-graph or a bony multi-graph. In particular, $K(G)$ carries $n$ ergodic SRB measures $\nu_{\gamma_{G,i}}$, $i=1, \dots,n$, given by (\ref{meas}).
\end{corollary}

\section{Thermodynamic properties of invariant graphs}
Take an skew product $G \in \mathcal{U}$ given by Theorem \ref{thm000}. Here, we discuss thermodynamic properties of invariant graphs $\Gamma_{G,i}$, $i=1, \dots , n$,
given by Theorem \ref{thm88}, and hence the bony multi-graph $K(G)$.

Let $(X, d)$ be a non-empty, compact metric space and $T : X \to X$
be a continuous transformation. For each $n \in \mathbb{N}$ define for $x, y \in X$ the metric
\begin{equation*}
  d_n(x, y) := \max \{d(T^i(x), T^i(y)) : 0 \leq i < n \} .
\end{equation*}
Given some $\epsilon > 0$, a subset $\emptyset \neq E \subseteq K \subseteq X$ is called $(\epsilon, n)$-\emph{separated} in $K$, if
\begin{equation*}
\inf \{ d_n(x, y) : x \neq y \in E \} \geq \epsilon.
\end{equation*}
In addition, $E \subseteq K$ is called \emph{maximally} $(\epsilon, n)$-\emph{separated} in $K$, if for all $z \in K$ the
set $E \cup \{ z \}$ is not $(\epsilon, n)$-separated anymore. Using Zorn's lemma, for every non-empty subset
$K \subset X$ there exists a maximally $(\epsilon, n)$-separated set $E \subset K$.

In what follows, the set of all $T$-invariant probability measures is denoted by
$\mathcal{M}_T(X)$. Moreover, we denote by $\mathcal{E}_T (X) \subseteq \mathcal{M}_T(X)$ the set of all $T$-invariant, ergodic probability measures on $X$.
Recall for $\mu \in \mathcal{M}_T(X)$, we denote by
$h_\mu(T)$ the measure-theoretic entropy of $T$ with respect to $\mu$. For
dynamical systems $(X, T)$ the quantity $h_{top}(T)$ denotes the topological entropy
of $(X, T)$, and one has
\begin{equation*}
 h_{top}(T)=\sup \{h_\mu(T): \mu \in \mathcal{M}_T(X)\}.
\end{equation*}
Let $(X, T)$ be a dynamical system and $\phi : X \to \mathbb{R}$ be an arbitrary
function. For every subset $K \subseteq X$, define
\begin{equation*}
  P_K(T, \phi) := \lim_{\epsilon \to 0}\limsup_{n \to \infty}\frac{1}{n}\log \sup_{E}\sum_{x \in E}\exp \sum_{i=0}^{n-1}\phi (T^i(x)),
\end{equation*}
where the supremum is taken over all $(\epsilon, n)$-separated sets in $K$.

The variational principle for the topological
pressure of continuous functions was proven in \cite{W}: One has
\begin{equation*}
  P_X(T, \phi) = \sup \{h_\mu(X) + \int_X \phi d\mu\},
\end{equation*}
where the supremum is taken over all ergodic $T$-invariant Borel probability measures
$\mu$ on $X$. Here $h_\mu(T)$ denotes the measure theoretic entropy of $T$ with respect to $\mu$.

Take an skew product $G\in \mathcal{U}$ given by Theorem \ref{thm000}.
We recall that the base map $S$ is an extension of expanding circle map $\omega: \mathbb{T}^1 \to \mathbb{T}^1$, $\omega(t)=4t \ (mod 1)$.
In the setting of uniformly expanding maps, equilibrium states always exist and they
are unique SRB measures if the potential is H\"{o}lder continuous and the dynamics is topologically exact, see \cite[Theorem 12.1]{VO}.
Clearly $\omega$ is a topologically exact uniformly expanding map and hence it admits a unique SRB equilibrium measure $\nu^+_\phi$ for each H\"{o}lder continuous potential $\phi$.
Moreover, it is supported on the whole $\mathbb{T}^1$.

As we have seen in Subsection 4.2, there exists a semiconjugacy $\textbf{p} : \Omega \to \mathbb{T}^1$
sending each sequence $\textbf{t}=(\ldots, t_{-n}, \ldots, t_{-1}, t_0)\in \Omega$ to its term $t_0$ of order zero.
Given an ergodic measure $\mu^+$ defined on
Borel subsets of $\mathbb{T}^1$ there exists a unique measure $\mu$ defined on Borel
subsets of $\Omega$ such that $\mu^+=\textbf{p}_\ast \mu$.

Here, we investigate the relation between entropy and topological pressure of the expanding circle map $\omega$ and its extension $S$. First, we need the following results.

Assume $f:M\longrightarrow M$ and $\tilde{f}:\tilde{M}\longrightarrow \tilde{M}$ are two continuous transformations of compact topological spaces $M$ and $\tilde{M}$, respectively.
If $f$ is a topological factor of $\tilde{f}$, then $h_{top}(f)\leq h_{top}(\tilde{f})$, see \cite{VO}.
Also, Ledrappier-Walter's formula states the relation between metric entropy of $\tilde{f}:\tilde{M}\to \tilde{M}$ and its topological factor $f:M\to M$:
\begin{theorem} $(\mathbf{Ledrappier-Walter's \ formula}).$\cite{LW}
  Let $\tilde{M}$ and $M$ be compact metric spaces and let $\tilde{f}:\tilde{M}\longrightarrow \tilde{M}$, $f:M\longrightarrow M$ and $\pi:\tilde{M}\longrightarrow M$ be continuous maps such that $\pi$ is surjective and $\pi\circ\tilde{f}=f\circ\pi$, then
$$\sup_{\tilde{\mu};\pi_{*}\tilde{\mu}=\mu}h_{\tilde{\mu}}(\tilde{f})=h_{\mu}(f)+\int h_{top}(\tilde{f},\pi^{-1}(y))d\mu(y).$$
\end{theorem}
\begin{lemma}\label{l}
The following two facts hold:
\begin{enumerate}
  \item [$(1)$] For any probability measures $\mu$ and $\mu^+$ invariant under $S$ and $\omega$, respectively, with $\mu^+=\textbf{p}_\ast \mu$, one has that $h_{\mu^+}(\omega)=h_\mu(S)$.
  \item [$(2)$] Let $\phi$ be a H\"{o}lder continuous potential for the expanding circle map $\omega$. Then $P_\omega(\phi)=P_S(\psi)$, where $\psi= \phi \circ \textbf{p}$.
\end{enumerate}
\end{lemma}
\begin{proof}
(1) Since
\begin{equation*}
  \textbf{p}^{-1}(t)=\{(\ldots, t_{-n}, \ldots, t_{-1}, t_0)\in \Omega: t_0=t\}
\end{equation*}
we observe that $h_{top}(\omega, \textbf{p}^{-1}(t))=0$
 for every $t \in \mathbb{T}^1$ because we can choose a subset of $\textbf{p}^{-1}(t)$
with finite cardinality as $n$-generator for every $n \in \mathbb{N}$.
Thus, we apply the Ledrappier-Walter's formula to conclude that $h_{\mu^+}(\omega)=h_\mu(S)$, see Section 3.3 of \cite{RV} for more details.

(2) First we observe that if $\phi$ is H\"{o}lder continuous then $\phi \circ \textbf{p}$ is also H\"{o}lder continuous.
Since $h_{top}(\omega, \textbf{p}^{-1}(t))=0$
 for every $t \in \mathbb{T}^1$, by the previous lemma and applying Ledrappier-Walter's formula, we have
 \begin{eqnarray*}
   P_\omega(\phi)\leq P_S(\psi) &=& \sup\{h_\mu(S)+\int \psi d\mu\}  \\
   &=& \sup\{h_{\mu^+}(\omega)+\int h_{top}(\omega, \textbf{p}^{-1}(t))d\mu^+(t)+\int \phi d\mu^+\} \\
    &=& P_\omega(\phi),
 \end{eqnarray*}
 where the supremum is taken over all probability measures $\mu$ invariant under $S$.
 Note that given an invariant measure $\mu^+$ defined on
Borel subsets of $\mathbb{T}^1$ there exists a unique invariant measure $\mu$ defined on Borel
subsets of $\Omega$ such that $\mu^+=\textbf{p}_\ast \mu$.
\end{proof}
By Theorem \ref{thm000}, there exists $S$-invariant sets $\Omega_i(G)\subseteq \Omega$, $i=1, \dots,n$, such that $\nu(\Omega_i(G))=1$
and measurable functions $\gamma_{G,i}:\Omega_i(G)\to I_i$ such that $\Gamma_{G,i}$, the graphs of $\gamma_{G,i}$, are invariant under $G$.
Furthermore $\Gamma_{G,i}$, $i=1, \dots,n$, are attracting $(G,\nu)$ invariant graphs.
Let us take the skew products
\begin{equation}\label{ss}
G_i:= G|_{\Omega \times I_i}, \ G_i(\textbf{t}, x)= (S(\textbf{t}), g(\textbf{t}, x)), \ 1\leq i\leq n.
\end{equation}
 By Lemma \ref{lem99}, we extend measurable functions $\gamma_{G,i}$ to the whole space $\Omega$ so that the extended functions are upper semicontinuous. We also write the extended functions by the same notation $\gamma_{G,i}$.

Recently, a lot of attention has been done to extend the definition of pressure to not necessarily
continuous functions $\phi$, and to prove a corresponding variational principle.
A variational principle for sub-additive, upper semi-continuous sequences of functions was established
in \cite{CFH} and \cite{BF}.
This result was recently generalized in \cite{FH} for weighted topological
pressure on systems with upper semi-continuous entropy mapping.
Then, Rauch \cite{R} extended the original definitions of pressure to discontinuous $\phi$, and compare
them to the classical ones. Furthermore he determined several classes of functions,
which admit variational inequalities and principles.

Following \cite{R}, we call $\phi$ to be \emph{quasi-integrable with respect to} $\mu$, if either
$\int_X \phi^+ d\mu <\infty$ or $\int_X \phi^- d\mu <\infty$, where $\phi^+:=\max(\phi,0)$ and $\phi^+:=\max(- \phi,0)$.
The set of all measurable $\phi : X \to \mathbb{R}$, which are quasi-integrable
for all $\mu \in \mathcal{M}_T (X)$, is defined by $Q_T (X)$. We call $\phi \in Q_T (X)$ \emph{quasi-integrable
with respect to} $T$.
Let be $\phi \in Q_T (X)$. The function $\phi$ is \emph{upper semi-continuous with
respect to} $T$, if the following holds: If $\{ \mu_n \}_{n\in \mathbb{N}}$
is a sequence of atomic probability
measures $\mu_n$ such that there exists a $\mu \in \mathcal{M}_T (X)$ satisfying $\mu_n \to \mu$
in the weak$^\ast$-topology, then
\begin{equation*}
  \limsup_{n \to \infty}\int_X \phi d\mu_n \leq \int_X \phi d\mu.
\end{equation*}
The set of all upper semi-continuous functions with respect to $T$ is denoted by
$U_T (X) \subseteq Q_T (X)$, see \cite{R} for more details.
\begin{theorem}\cite[Theorem C]{R}
Let $(X, T)$ be a dynamical system satisfying
$h_{top}(T) < \infty$. If $\phi : X \to \mathbb{R}$ is upper semi-continuous with respect to $T$ (see the above
definition), then one has
\begin{equation*}
  P_T(X,\phi)=\sup\{h_\mu(T)+ \int_X \phi d\mu \},
\end{equation*}
where the supremum is taken over all $T$-invariant Borel probability measures $\mu$ on
$X$.
\end{theorem}
A function $\phi : X \to \mathbb{R}$ is called \emph{upper semi-continuous}, if
$\{ x \in X : \phi(x) < c\}$ is an open set for every $c \in \mathbb{R}$. By definition every upper
semi-continuous function is also Borel measurable. We denote the set of all upper
semi-continuous functions $\phi : X \to \mathbb{R}$ by $U(X)$. As $X$ is compact, every $\phi \in U(X)$
is bounded from above (see for example \cite{AB} Theorem 2.43). This immediately yields
$U(X) \subseteq Q_T (X)$. In addition, by Proposition 6 of \cite{R}, one has $U(X) \subseteq U_T (X)$ for every continuous mapping $T : X \to X$.
\begin{corollary}\cite[Corollary~6]{R}\label{upper}
Let $h_{top}(T) < \infty$ and $\phi \in U(X)$. Then one has
\begin{equation*}
  P_T(X,\phi)=\sup\{h_\mu(T)+ \int_X \phi d\mu \},
\end{equation*}
where the supremum is taken over all $T$-invariant Borel probability measures $\mu$ on
$X$.
\end{corollary}
Also, we recall the following proposition from \cite{R}.
\begin{proposition}
Let $\{\mu_n\}_{n \in \mathbb{N}}$  be a sequence of Borel probability measures with limit measure $\mu$ in the weak$^*$ topology.
Then one has for $\varphi \in U(X)$
\begin{equation*}
  \limsup_{n \to \infty}\int_X \varphi d\mu_n \leq \int_X \varphi d\mu.
\end{equation*}
\end{proposition}\label{p1}
\begin{proposition}\label{eq}
Consider the skew products $G_i= G|_{\Omega \times I_i}$, $1\leq i\leq n$, given by (\ref{ss})
and let $\phi_i: \Omega \times I_i \to \mathbb{R}$ be a H\"{o}lder continuous potential.
Take $\psi_i : \Omega \to \mathbb{R}$ defined by $\psi_i(\textbf{t}) = \phi_i(\textbf{t}, \gamma_{G,i}(\textbf{t}))$.
Then there exist some equilibrium state $\mu_{\psi_i}$ for $(S, \psi_i)$.
\end{proposition}
\begin{proof}
By Lemma \ref{lem99}, $\psi_i$ is upper semicontinuous.
By Corollary \ref{upper}, one has that
\begin{equation*}
  P_S(\Omega,\psi_i)=\sup\{h_\mu(S)+ \int_\Omega \psi_i d\mu \},
\end{equation*}
where the supremum is taken over all $S$-invariant Borel probability measures $\mu$ on
$\Omega$.
Since the expanding circle map $\omega$ is a $C^2$ local diffeomorphism then
the natural extension $S$ is locally Lipschitz continuous, i.e., given
$\textbf{t} \in \Omega$ there exists a neighborhood $V_\textbf{t}$
such that for every $\textbf{t}_1, \textbf{t}_2 \in S(V_{\textbf{t}})$ we have
\begin{equation*}
  d_\Omega(S^{-1}(\textbf{t}_1),S^{-1}(\textbf{t}_2))\leq \sigma(\textbf{t})d_{\Omega}(\textbf{t}_1, \textbf{t}_2)
\end{equation*}
where $\sigma(\textbf{t})=\|D\omega^{-1}\| \circ \textbf{p}(\textbf{t})$. In particular, $S$ is a Ruelle expanding map and hence it is expansive.
By Corollary 9.2.17 of \cite{VO}, the entropy function of an expansive transformation in a compact
metric space is upper semi-continuous. By Proposition \ref{p1} and upper semicontinuity of $\psi_i$, the function $\mu \mapsto \int_\Omega \psi_i d\mu$ is also upper semicontinuous.
 By these facts the map $\mu \mapsto h_\mu(S)+ \int_\Omega \psi_i d\mu$
is upper semicontinuous and hence, there exists some equilibrium state $\mu_{\psi_i}$.
Indeed, let $\{\mu_n\}_{n \in \mathbb{N}}$  be a sequence of Borel probability invariant measures such that
\begin{equation*}
  h_{\mu_n}(S)+ \int_{\Omega} \psi_i d\mu_n \quad \text{converges \ to} \quad P_S(\Omega,\psi_i).
\end{equation*}
Since the space invariant Borel probability measures is compact, there exists some accumulation point $\mu_{\psi_i}$.
By this fact and upper semicontinuity of $\mu \mapsto h_\mu(S)+ \int_\Omega \psi_i d\mu$, we get
\begin{equation*}
 h_{\mu_{\psi_i}}(S)+ \int_{\Omega} \psi_i d\mu_{\psi_i} \geq \liminf_{n \to \infty}h_{\mu_n}(S)+ \int_{\Omega} \psi_i d\mu_n=P_S(\Omega,\psi_i),
\end{equation*}
so $\mu_{\psi_i}$ is an equilibrium state, as stated.
\end{proof}
Let $\phi: \Omega \times I \to \mathbb{R}$ be a H\"{o}lder continuous potential that does not depend on the fiber.
This means that the function
$\phi(\textbf{t}, .) : \Omega \to \mathbb{R}$ is constant, for each $\textbf{t} \in \Omega$ fixed.
Take the restriction $\phi_i:=\phi |_{\Omega \times I_i}$, $i=1, \dots, n$. Clearly, $\phi_i$ is also a H\"{o}lder continuous potential.
Hence, as above, the potential $\phi_i$ induces an upper
semicontinuous potential $\psi_i : \Omega \to \mathbb{R}$ defined by $\psi_i(\textbf{t}) = \phi_i(\textbf{t}, \gamma_{G,i}(\textbf{t}))$.

We recall that the expanding circle map $\omega$ possesses an absolutely continuous invariant ergodic measure $\nu^+$
which is equivalent to Lebesgue. Then, $(\Omega, S)$ has an invariant ergodic measure $\nu$ inherited from the invariant
measure $\nu^+$, i.e. $\nu^+=\textbf{p}_\ast \nu$.
\begin{proposition}\label{p00}
Let $\phi: \Omega \times I \to \mathbb{R}$ be a H\"{o}lder continuous potential and take $\phi_i=\phi |_{\Omega \times I_i}$, $i=1, \dots, n$.
Consider the skew product $G\in \mathcal{U}$ and $G_i= G|_{\Omega \times I_i}$, $1\leq i\leq n$, given by (\ref{ss}).
Assume the measure $\mu_{\psi_i}$, given by Proposition \ref{eq}, is ergodic and it is not singular with respect to
the invariant ergodic measure $\nu$.
Then $\mu_{\phi_i}=\mu_{\psi_i} \circ (id \times \gamma_{G,i})^{-1}$ is an equilibrium state associated to $(G_i, \phi_i)$ which is supported
on the maximal attractors $\Delta_i(G)$.
\end{proposition}
\begin{proof}
Take the measure $\mu_{\psi_i}$, given by Proposition \ref{eq}, which is ergodic and it is not singular with respect to
the invariant ergodic measure $\nu$. As both measures are $S$-invariant and ergodic, this implies
that these measures coincide.

Since $\mu_{\psi_i}$ is ergodic, so $\mu_{\phi_i}$ is also ergodic. In particular, since $\mu_{\psi_i}$ is equivalent to $\nu$, so $\mu_{\phi_i}$ is supported
on the maximal attractors $\Delta_i(G)$.


By Lemma \ref{cor0}, $\lambda_i(\nu, G_i)$ is negative and $G_i$ satisfies the non-uniformly contraction condition, hence, for $\nu$- almost every $\textbf{t} \in \Omega$, $h_{top}(G_i, \pi^{-1}(\textbf{t}))=0$.
Since $\mu_{\psi_i}$ is equivalent to
$\nu$, thus, for $\mu_{\psi_i}$- almost every $\textbf{t} \in \Omega$, $h_{top}(G_i, \pi^{-1}(\textbf{t}))=0$.
Applying Ledrappier-Walter's formula, we have
 \begin{eqnarray*}
   P_S(\psi_i)\leq P_{G_i}(\phi_i) &=& \sup\{h_{\hat{\mu}}(G_i)+\int \phi_i d\hat{\mu}\}  \\
   &=& \sup\{h_{\mu}(S)+\int h_{top}(G_i, \pi^{-1}(\textbf{t}))d\mu(\textbf{t})+\int \psi_i d\mu\} \\
    &=& P_S(\psi_i).
 \end{eqnarray*}
Hence, $P_S(\psi_i)= P_{G_i}(\phi_i)$ and
$\mu_{\phi_i}$ is an equilibrium state for
$(G_i, \phi_i)$.
\end{proof}
\section{Skew products over the toral baker map}
Here, a toral baker map on $\mathbb{T}^2=\mathbb{T}^1 \times \mathbb{T}^1$ is defined by
\begin{align}\label{17}
H:\mathbb{T}^2\to \mathbb{T}^2, \ H(t,s)=(bt(mod 1),\frac{(s+[bs])}{b}),
\end{align}
for some positive integer $b$ and for each $\theta=(t,s)\in \mathbb{T}^2$.
This map is bijective but not continuous. Moreover, it preservers the Lebesgue measure $m$.
Hence, $(\mathbb{T}^2, \mathcal{B}, m, H)$ is a measure-preserving dynamical system, in the sense of Arnold \cite{AL},
where $\mathcal{B}$ is the Borel $\sigma$-algebra on $\mathbb{T}^2$.

Assume $\mathcal{F}_H \subset\mathcal{F}$ denotes the family of all skew product transformations $\mathbb{G}$ with the base map $H$ defined on $\mathbb{T}^2 \times \mathbb{I}$ of the form
\begin{equation}\label{skew00}
\mathbb{G}: \mathbb{T}^2 \times \mathbb{I} \to \mathbb{T}^2 \times \mathbb{I}, \ \mathbb{G}(\theta,x)=(H(\theta),g_{\theta}(x)),
\end{equation}
where $(\theta,x)\in \mathbb{T}^2 \times \mathbb{I}$, $\theta=(t,s)\in \mathbb{T}^2$
and the fiber maps $g_{\theta}$ depends on $\theta=(t,s)$ only through $t$, so we can write $g_{\theta}=g_{t}$, where $\mathcal{F}$ is given by Definition \ref{generalskew}.
By definition, $\mathbb{G}(\mathbb{T}^2 \times \mathbb{I}) \subset \mathbb{T}^2 \times \text{int}(\mathbb{I})$.


Note that the baker map $H$ is bijective but not continuous. Moreover, it is a canonical extension of the corresponding expanding circle map $\omega(t)=bt \ \text{mod 1}$, where $b$ is a positive integer.
Here, we assume that $b=4$.
\begin{lemma}
There exists a measure-theoretical isomorphism $h: (\mathbb{T}^2, H, m) \to(\Omega, S, \nu)$, where $m$ is the Lebesgue measure on $\mathbb{T}^2$.
\end{lemma}
\begin{proof}
We define $h:\mathbb{T}^2 \to \Omega$ by $h(\theta)=(\ldots,t_{-1},t_{0})$, for $\theta=(t,s)\in \mathbb{T}^2$ so that $(t_0,s_0)=(t,s)$ and for each
$j \in \mathbb{N}$, $(t_{-j},s_{-j})=H^{-j}(t_0,s_0)$. Then, it is easy to see that $h$ is a bijective map.
Indeed, the injectivity of the baker map $H$ implies the injectivity of $h$. Assume $\textbf{t}=(\ldots,t_{-1},t_{0})\in \Omega$. Take
\begin{equation*}
 t=t_0, \  s=\sum_{i=0}^\infty \frac{\lfloor 4 t_{-i-1} \rfloor}{4^{(i+1)}}, \ \text{and} \ \theta = (t,s).
\end{equation*}
Then $h(\theta)=\textbf{t}$ and hence, $h$ maps the two dimensional torus $\mathbb{T}^2$ onto the inverse limit space $\Omega$. Also, $h \circ H=S \circ h$.
It is easy to see that, $h$ is a measure-theoretical isomorphism between the two measure preserving dynamical systems $(\mathbb{T}^2, H, m)$ and $(\Omega, S, \nu)$, where $m$ is the Lebesgue measure on $\mathbb{T}^2$.
\end{proof}
Let us take the invariant graph $\widetilde{F}$ is given by (\ref{be}). It admits an extension $\mathbb{F}\in \mathcal{F}_H$ of the form (\ref{skew00}) and an extension $F$
given by (\ref{sss}) over the solenoid map $S$.
 \begin{theorem}\label{thm777}
There exists an open set $\mathcal{U}_H \subset \mathcal{F}_H$ of skew products over the baker map $H$ such that each skew product $\mathbb{G} \in \mathcal{U}_H$
admits an attracting multi-graph or an attracting bony multi-graph.
\end{theorem}
\begin{proof}
Take any skew product $\mathbb{G} \in \mathcal{F}_H$ over the baker map $H$ of the form (\ref{skew00}) sufficiently small to $\mathbb{F}$. Note that for iterates of $\mathbb{G}$ we denote
\begin{equation*}
\mathbb{G}^n(\theta,x)=(H^n(\theta),g_{H^{n-1}(\theta)}\circ \cdots \circ g_{\theta}(x))= (H^n(\theta),g_\theta^{n}(x)).
\end{equation*}
Since the fiber maps $g_{\theta}$ depend on $\theta=(t,s)$ only through $t$, so we can write $g^n_{\theta}(x)=g_{\omega^{n-1}(t)}\circ \dots \circ g_{t}(x)$.
Thus $g^n_{\theta}(x)=g^n_{\textbf{t}}$, for each $\textbf{t}=(\dots,t_{-1},t_0)\in \Omega$ with $t_0=t$. By this fact, for each skew product $G\in \mathcal{U}$
over the base map $S$ satisfying in the conclusion of Lemma \ref{cor0}, we
associate a skew product $\mathbb{G}$ over the baker map $H$  defined by $\mathbb{G}(\theta,x)=(H(\theta),g(\theta,x))=(H(\theta),g_t(x))$, $\theta=(t,s)\in \mathbb{T}^2$, where $g_t$, $t\in \mathbb{T}^1$, are the fiber maps of $G$. Therefore, there exist an open set $\mathcal{U}_H\subset \mathcal{F}_H$ and a one to one correspondence between the skew products $G\in \mathcal{U}$ over the base map $S$
 and the skew products $\mathbb{G}\in \mathcal{U}_H$ over the baker map $H$. Moreover, $(h \times id)\circ \mathbb{G}=G\circ (h \times id)$.

Given a skew product transformation $G\in \mathcal{U}$, take the invariant graphs $\gamma_{G,i}$ given by Theorem \ref{thm000} defined on the subset $\Omega_i(G)\subseteq \Omega$
with total measure. By Theorem \ref{thm88}, $\text{Cl}(\Gamma_{G,i})=\Delta_i(G)$, where $\Delta_i(G)$ is the maximal attractor given by (\ref{maxi}), $i=1, \dots,n$.

Now, for each $\theta=(t,s)\in \mathbb{T}^2$, define $\gamma_{\mathbb{G},i}(t,s) =\gamma_{G,i}(h(t,s))$.
We denote the graph of $\gamma_{\mathbb{G},i}$ by $\Gamma_{\mathbb{G},i}$.
Notice that the fiber map of $\mathbb{G}$ is constant along the stable leaves of the baker map $H$ given by vertical
fibers $\{t\} \times \mathbb{T}^1$, hence $\gamma_{\mathbb{G},i}(t,s)$ is constant along the stable leaves of $H$.

Take $I_i(\textbf{t})=\{\textbf{t}\}\times I_i$ for $\textbf{t}=(\cdots,t_{-1},t_0)\in \Omega$, and
\begin{equation*}
  I_i(\textbf{t},m,G)=g_{S^{-1}(\textbf{t})}\circ \cdots \circ g_{S^{-m}(\textbf{t})}(I_i)=g_{t_{-1}}\circ \cdots \circ g_{t_{-m}}(I_i),
\end{equation*}
for $i=1, \dots,n$. Then
\begin{equation*}
 \Delta_i(G)\bigcap I_i(\textbf{t})=\bigcap_{n\geq 0}I_i(\textbf{t},m,G).
\end{equation*}
Note that, by construction, the fiber maps of the skew products $G$ and $\mathbb{G}$ are the same.
Hence, by these observations and the previous lemma, $\text{Cl}(\Gamma_{\mathbb{G},i})$ is an attracting invariant graph or an attracting bony graph for $\mathbb{G}$.

Let us take
\begin{equation}\label{maxi1}
 A_{max}(\mathbb{G}):=\bigcap_{n\geq 0}\mathbb{G}^n(\mathbb{T}^2 \times I), \ \text{and} \ \Delta_i(\mathbb{G}):=\bigcap_{n\geq 0}\mathbb{G}^n(\mathbb{T}^2 \times I_i), \ 1\leq i\leq n.
\end{equation}
Then $\Delta_i(\mathbb{G})=\text{Cl}(\Gamma_{\mathbb{G},i})$, $1\leq i\leq n$, and thus the union $K(\mathbb{G}):=\bigcup_{i=1}^n \Delta_i(\mathbb{G})$ is an attracting multi-graph or bony multi-graph.
\end{proof}
\section{Data Availability}
Data sharing not applicable to this article as no datasets were generated or analysed during the current study.
 \section*{Conflict of interest}
 The authors declare that they have no conflict of interest.

\end{document}